\documentclass[12pt]{article}

\usepackage{amsthm, amsmath, amssymb, verbatim,graphicx}

\numberwithin{equation}{section}

\theoremstyle{plain}	     
\newtheorem{thm}{Theorem}[section] 
\newtheorem{cor}[thm]{Corollary}
\newtheorem{lem}[thm]{Lemma}
\newtheorem{prop}[thm]{Proposition}
     
\theoremstyle{definition}

\theoremstyle{remark} 
\newtheorem{rem}[thm]{Remark}
	
\newcommand{\disp}{\displaystyle}

\begin{document}
\title{A new form of the \\ generalized complete elliptic integrals
\footnote{This work was supported by MEXT/JSPS KAKENHI Grant (No. 24540218).}}
\author{Shingo Takeuchi \\
Department of Mathematical Sciences\\
Shibaura Institute of Technology
\thanks{307 Fukasaku, Minuma-ku,
Saitama-shi, Saitama 337-8570, Japan. \endgraf
{\it E-mail address\/}: shingo@shibaura-it.ac.jp \endgraf
}}
\date{}

\maketitle

\begin{abstract}
Generalized trigonometric functions are applied to
the Legendre-Jacobi standard form of complete elliptic 
integrals, and a new form of 
the generalized complete elliptic integrals of the Borweins
is presented. According to the form, it can be easily shown that   
these integrals have similar properties to the classical ones.
In particular, it is possible to establish a computation
formula of the generalized $\pi$ in terms of the arithmetic-geometric mean,
in the classical way 
as the Gauss-Legendre algorithm for $\pi$ by Salamin and Brent. 
Moreover, an elementary new proof of Ramanujan's cubic transformation is 
also given.
\end{abstract}



\section{Introduction}

Complete elliptic integrals of the first kind and of the second kind
\begin{align*}
K(k)
&=\int_0^{\frac{\pi}{2}} \frac{d\theta}{\sqrt{1-k^2\sin^2{\theta}}}
=\int_0^1 \frac{dt}{\sqrt{(1-t^2)(1-k^2t^2)}},\\
E(k)
&=\int_0^{\frac{\pi}{2}} \sqrt{1-k^2\sin^2{\theta}}\,d\theta
=\int_0^1 \sqrt{\frac{1-k^2t^2}{1-t^2}}\,d\theta
\end{align*}
are classical integrals which have helped us, for instance, 
to evaluate the length of curves and
to express exact solutions
of differential equations.

In this paper we give a generalization of complete elliptic integrals
as an application of generalized trigonometric functions.
For this, we need the generalized sine function $\sin_p{\theta}$ 
and the generalized $\pi$ denoted by $\pi_p$,
where $\sin_p{\theta}$ is the inverse function of 
$$\sin_p^{-1}{\theta}:=\int_0^\theta \frac{dt}{(1-t^p)^{\frac1p}},\quad 
0 \leq \theta \leq 1,$$
and $\pi_p$ is the number defined by
$$\pi_p:=2\sin_p^{-1}{1}=2\int_0^1\frac{dt}{(1-t^p)^{\frac1p}}
=\frac{2\pi}{p\sin{\frac{\pi}{p}}}.$$
Clearly, $\sin_2{\theta}=\sin{\theta}$ and $\pi_2=\pi$.
These two appear in the eigenvalue problem of 
one-dimensional $p$-Laplacian:
$$-(|u'|^{p-2}u')'=\lambda |u|^{p-2}u,\quad u(0)=u(1)=0.$$
Indeed, the eigenvalues are given as $\lambda_n=(p-1)(n\pi_p)^p,\ n=1,2,3,\ldots$,
and the corresponding eigenfunction to $\lambda_n$ is 
$u_n(x)=\sin_p{(n \pi_p x)}$ for each $n$.
There are a lot of literature on generalized trigonometric functions
and related functions. See \cite{Bu,BE,DoR,EGL,E,LE,Li,LP,LP2,Sh}
for general properties as functions;
\cite{DEM,DoR,DM,LE,N,T} for applications to differential equations
involving $p$-Laplacian;
\cite{BBCDG,BE,BL,EGL,EGL2,LE,T2} for basis properties for sequences 
of these functions. 

Now, applying $\sin_p{\theta}$ and $\pi_p$ to the complete elliptic functions, 
we define the
\textit{complete $p$-elliptic integrals of the first kind $K_p(k)$} and 
\textit{of the second kind $E_p(k)$}: for $p \in (1,\infty)$ and $k \in [0,1)$
\begin{align}
\label{eq:CpEIK}
K_p(k)
&:=\int_0^{\frac{\pi_p}{2}} \frac{d\theta}{(1-k^p\sin_p^p{\theta})^{1-\frac1p}}
=\int_0^1 \frac{dt}{(1-t^p)^{\frac1p}(1-k^pt^p)^{1-\frac1p}},\\
\label{eq:CpEIE}
E_p(k)
&:=\int_0^{\frac{\pi_p}{2}} (1-k^p\sin_p^p{\theta})^{\frac1p}\,d\theta
=\int_0^1 \left(\frac{1-k^pt^p}{1-t^p}\right)^{\frac1p}\,dt.
\end{align}
Here, each second equality of the definitions is obtained by setting 
$\sin_p{\theta}=t$. 
It is easy to see that for $p=2$ these integrals are equivalent to the classical
complete elliptic integrals $K(k)$ and $E(k)$.

It is worth pointing out that  
the Borweins \cite[Section 5.5]{BB3} define 
the \textit{generalized complete elliptic 
integrals of the first} and \textit{of the second kind} by
\begin{align}
\label{eq:Ks}
\mathrm{K}_s(k)
&:=\frac{\pi}{2}F\left(\frac12-s,\frac12+s;1;k^2\right),\\
\label{eq:Es}
\mathrm{E}_s(k)
&:=\frac{\pi}{2}F\left(-\frac12-s,\frac12+s;1;k^2\right)
\end{align}
for $|s|<1/2$ and $0 \leq k <1$,
where $F(a,b;c;x)$ denotes the Gaussian hypergeometric function
(see Section 2 for the definition). Note that 
$\mathrm{K}_0(k)=K(k)$ and $\mathrm{E}_0(k)=E(k)$.
According to Euler's integral representation (see  
\cite[Theorem 2.2.1]{AAR} or \cite[p.\,293]{WW}), we have
\begin{align*}
\mathrm{K}_s(k)
&=\frac{\cos{\pi s}}{2s+1}\int_0^1
\frac{dt}{(1-t^{\frac{2}{2s+1}})^{\frac{2s+1}{2}}
(1-k^2t^{\frac{2}{2s+1}})^{1-\frac{2s+1}{2}}},\\
\mathrm{E}_s(k)
&=\frac{\cos{\pi s}}{2s+1}\int_0^1
\left(\frac{1-k^2t^{\frac{2}{2s+1}}}{1-t^{\frac{2}{2s+1}}}\right)^{\frac{2s+1}{2}}\,dt.
\end{align*}
Thus 
$$\mathrm{K}_s(k)=\frac{\pi}{\pi_p}K_p(k^\frac{2}{p}),\quad
\mathrm{E}_s(k)=\frac{\pi}{\pi_p}E_p(k^\frac{2}{p}),
$$
where $p=2/(2s+1)$.
To show this, one may use Proposition \ref{prop:hypergeometricexpression} below
instead of integral representations. 
Anyway, we emphasize that the complete $p$-elliptic integrals 
\eqref{eq:CpEIK} and \eqref{eq:CpEIE} give representations 
of generalized complete elliptic integrals in the Legendre-Jacobi standard form
with generalized trigonometric functions.
The advantage of using the complete $p$-elliptic integrals lies in the fact that 
it is possible to prove formulas of the generalized complete elliptic 
integrals simply as well as that of the classical complete elliptic integrals.
For example, we have known the following Legendre relation between $K(k)$
and $E(k)$ (see \cite{AAR,BB3,EMOT,WW}).
\begin{equation}
\label{eq:legendrep=2}
K'(k)E(k)+K(k)E'(k)-K(k)K'(k)=\frac{\pi}{2},
\end{equation}
where $k':=\sqrt{1-k^2},\ K'(k):=K(k')$ and $E'(k):=E(k')$.
For this we can show the following relation between $K_p(k)$ and $E_p(k)$.
\begin{thm}
\label{thm:p-legendre}
For $k \in (0,1)$
\begin{equation}
\label{eq:p-legendre}
K_p'(k)E_p(k)+K_p(k)E_p'(k)-K_p(k)K_p'(k)=\frac{\pi_p}{2},
\end{equation}
where $k':=(1-k^p)^{\frac1p},\ K_p'(k):=K_p(k')$ and $E_p'(k):=E_p(k')$.
\end{thm}
In fact, it is known that
$\mathrm{K}_s(k)$ and $\mathrm{E}_s(k)$ also satisfy
the similar relation \eqref{eq:legendre-borwein} below to \eqref{eq:p-legendre}, 
which follows from Elliott's identity \eqref{eq:elliott} below.
In contrast to this, our approach with generalized trigonometric functions
seems to be more elementary and self-contained.

As an application of generalized $p$-elliptic integrals,
we establish a computation formula of $\pi_p$.
Let us first mention the case $p=2$. In that case, consider
the sequences $\{a_n\}$ and $\{b_n\}$ satisfying $a_0=a,\ b_0=b$,
where $a \geq b>0$, and
$$a_{n+1}=\frac{a_n+b_n}{2},\quad b_{n+1}=\sqrt{a_nb_n},\quad n=0,1,2,\ldots.$$
It is easily checked that these sequences converge to a common limit,
the \textit{arithmetic-geometric mean of $a$ and $b$}, denoted by $M(a,b)$.
It is well-known that on May 30th in 1799 
Gauss discovered a celebrated relation between 
$M(a,b)$ and $K(k)$ (precisely, a relation between 
$M(1,\sqrt{2})$ and the lemniscate integral): 
\begin{equation}
\label{eq:gauss}
K(k)=\frac{\pi}{2}\frac{1}{M(1,k')},
\end{equation}
where $k'=\sqrt{1-k^2}$. Combining \eqref{eq:legendrep=2} and \eqref{eq:gauss}
with $k=k'=1/\sqrt{2}$, Salamin \cite{Sa} and Brent \cite{Br} independently
established the following formula (see also \cite{AAR,BB3} for the proof).
\begin{equation}
\label{eq:pi}
\pi=\frac{2M\left(1,\dfrac{1}{\sqrt{2}}\right)^2}
{\disp 1-\sum_{n=0}^\infty 2^n(a_n^2-b_n^2)},
\end{equation}
where the initial data of $\{a_n\}$ and $\{b_n\}$ are $a=1$ and $b=1/\sqrt{2}$.
This is known as a fundamental formula to Gauss-Legendre algorithm,
or Salamin-Brent algorithm, for computing the value of $\pi$.

Owing to \eqref{eq:pi}, it is natural to try to establish 
a computation formula of $\pi_p$.
In addition we are interested in finding such an elementary way of 
its construction as Salamin and Brent.

In the present paper, we give the formula only for the case $p=3$.
We prepare notation for stating results. 
Let $a \geq b>0$, and assume that 
$\{a_n\}$ and $\{b_n\}$ are sequences
satisfying $a_0=a, b_0=b$ and 
\begin{equation}
\label{eq:sequence}
a_{n+1}=\frac{a_n+2b_n}{3}, \quad b_{n+1}=\sqrt[3]{\frac{(a_n^2+a_nb_n+b_n^2)b_n}{3}},\quad n=0,1,2,\ldots.
\end{equation}
It is easy to see that both the sequences converge to 
the same limit as $n \to \infty$, denoted by $M_3(a,b)$.
Then, we obtain 
\begin{thm}
\label{thm:p-agm}
Let $0 \leq k<1$. Then
$$K_3(k)=\frac{\pi_3}{2}\frac{1}{M_3(1,k')},$$
where $k'=\sqrt[3]{1-k^3}$.
\end{thm}

Actually, Theorem \ref{thm:p-agm} is identical to the result
of the Borweins \cite[Theorem 2.1 (b)]{BB2} (with some trivial typos).
In either proof, it is essential to show Ramanujan's cubic transformation
(Lemma \ref{lem:ramanujan} below) . We will give a new proof
for this by more elementary calculation with properties of $K_3(k)$.

By Theorems \ref{thm:p-legendre} and \ref{thm:p-agm} 
we obtain the following formula of $\pi_3$.
\begin{thm}
\label{thm:main}
Let $a=1$ and $b=1/\sqrt[3]{2}$. 
Then
$$\pi_3=\frac{\displaystyle 2M_3\left(1,\frac{1}{\sqrt[3]{2}}\right)^2}
{\displaystyle 1-2\sum_{n=1}^\infty 3^n (a_n+c_n)c_n},$$
where $\{a_n\}$ and $\{b_n\}$ are the sequences \eqref{eq:sequence}
and $c_{n}:=\sqrt[3]{a_n^3-b_n^3}$.
\end{thm}

By the theory of theta functions, 
the Borweins \cite[Section 3]{BB2} give three iterations for $\pi$.
One of them is obtained from the same sequences as \eqref{eq:sequence}:
\begin{equation}
\label{eq:BB}
\pi=\dfrac{3 M_3\left(1,\left(\dfrac{\sqrt{3}-1}{2}\right)'\right)^2}
{\displaystyle 1-\sum_{n=0}^\infty 3^{n+1}(a_n^2-a_{n+1}^2)}.
\end{equation}
However, note that the initial data $a=1$ and $b=((\sqrt{3}-1)/2)'$
of \eqref{eq:BB}
is different from our initial data $a=1$ and $b=1/\sqrt[3]{2}$.

It is a simple matter to obtain other formulas for $\pi_3$ 
if we combine $\pi_3=4\sqrt{3}\pi/9=2.418\cdots$ with \eqref{eq:pi}
or \eqref{eq:BB}. 
The former converges quadratically 
to $\pi_3$ and the latter does cubically. 
On the other hand, our formula in Theorem \ref{thm:main} 
converges cubically to $\pi_3$ (Table 1).
\begin{table}[htbp]
\centering
\begin{tabular}{|c|c|c|} \hline
& 25 digits & Error \\ \hline
$q_1$ & $2.418399152309345558425031$ & $2.9449 \times 10^{-12}$ \\ \hline
$q_2$ & $2.418399152312290467458771$ & $4.0425 \times 10^{-40}$ \\ \hline
$q_3$ & $2.418399152312290467458771$ & $1.0367 \times 10^{-124}$ \\ \hline
$q_4$ & $2.418399152312290467458771$ & $1.8728 \times 10^{-379}$ \\ \hline
\end{tabular}
\caption{Convergence of  $q_m$ to $\pi_3$, where 
$q_m:=\frac{2a_{m+1}^2}
{1-2\sum_{n=1}^m 3^n (a_n+c_n)c_n}$.}
\end{table}
However, we are not interested in such trivial formulas obtained from
those of $\pi$, and it is not our purpose to study the speed of convergence 
and we will not develop this point here.

This paper is organized as follows.
In Section 2 we have compiled some basic facts of 
complete $p$-elliptic integrals. In particular
we show Legendre's relation for $K_p(k)$ and $E_p(k)$ 
(Theorem \ref{thm:p-legendre})
and observe relationship between the complete $p$-elliptic integrals
and the Gaussian hypergeometric functions (Theorem \ref{thm:p-agm}).
Section 3 establishes a computation formula of $\pi_p$
with $p=3$
as an application of complete $p$-elliptic integrals  (Theorem \ref{thm:main}). 
In particular, we give an elementary proof of Ramanujan's
cubic transformation by using our representation of integrals 
(Lemma \ref{lem:ramanujan}).
 

\section{Complete $p$-Elliptic Integrals}


In this section, we present some basic properties of
complete $p$-elliptic integrals $K_p(k)$ and $E_p(k)$.

Let $1<p<\infty$. We repeat the definition of 
complete $p$-elliptic integrals of the first kind $K_p(k)$ and 
of the second kind $E_p(k)$: for $k \in [0,1)$, 
\begin{align}
K_p(k)
&:=\int_0^{\frac{\pi_p}{2}} \frac{d\theta}{(1-k^p\sin_p^p{\theta})^{1-\frac1p}}
=\int_0^1 \frac{dt}{(1-t^p)^{\frac1p}(1-k^pt^p)^{1-\frac1p}}, \tag{\ref{eq:CpEIK}} \\
E_p(k)
&:=\int_0^{\frac{\pi_p}{2}} (1-k^p\sin_p^p{\theta})^{\frac1p}\,d\theta
=\int_0^1 \left(\frac{1-k^pt^p}{1-t^p}\right)^{\frac1p}\,dt, \tag{\ref{eq:CpEIE}}
\end{align}
where $\sin_p{\theta}$ and $\pi_p$ have been defined in the Introduction.
As is traditional, we will use the notation 
$k':=(1-k^p)^{\frac1p}$.
The variable $k$ is often called the \textit{modulus}, 
and $k'$ is the \textit{complementary modulus}.
The \textit{complementary integrals} $K_p'(k)$ and $E_p'(k)$
are defined by $K_p'(k):=K_p(k')$ and $E_p'(k):=E_p(k')$.

Let $\cos_p{\theta}:=(1-\sin_p^p{\theta})^{\frac1p}$.
The following formulas will be frequently used: 
$$\sin_p^p{\theta}+\cos_p^p{\theta}=1,$$
$$\frac{d}{d\theta}(\sin_p{\theta})=\cos_p{\theta},\quad
\frac{d}{d\theta}(\cos_p^{p-1}{\theta})=-(p-1)\sin_p^{p-1}{\theta}.$$
If $p=2$ then $\sin_p{\theta},\ \cos_p{\theta}$ and $\pi_p$ 
coincide with the usual $\sin{\theta},\ \cos{\theta}$ and $\pi$, respectively,
so that these properties above are familiar.


The functions $K_p(k)$ and $E_p(k)$ satisfy a system of differential equations.
\begin{prop}
\label{prop:p-differential}
$$\frac{dE_p}{dk}=\frac{E_p-K_p}{k},\quad
\frac{dK_p}{dk}=\dfrac{E_p-(k')^pK_p}{k(k')^p}.$$
\end{prop}

\begin{proof}
Differentiating $E_p(k)$ we have
\begin{align*}
\frac{dE_p}{dk}
&=\int_0^{\frac{\pi_p}{2}}
\frac{d}{dk}(1-k^p\sin_p^p{\theta})^{\frac1p}\,d\theta\\
&=\int_0^{\frac{\pi_p}{2}}
\dfrac{-k^{p-1}\sin_p^p{\theta}}{(1-k^p\sin_p^p{\theta})
^{1-\frac1p}}\,d\theta\\
&=\frac{1}{k} \int_0^{\frac{\pi_p}{2}}
\dfrac{1-k^p\sin_p^p{\theta}}{(1-k^p\sin_p^p{\theta})^{1-\frac1p}}
\,d\theta
-\frac{1}{k} \int_0^{\frac{\pi_p}{2}}
\dfrac{d\theta}{(1-k^p\sin_p^p{\theta})^{1-\frac1p}}\\
&=\frac{1}{k} (E_p-K_p).
\end{align*}

Next, for $K_p(k)$ 
\begin{align}
\label{eq:dK}
\frac{dK_p}{dk}
=\int_0^{\frac{\pi_p}{2}}
\dfrac{(p-1)k^{p-1}\sin_p^p{\theta}}
{(1-k^p\sin_p^p{\theta})^{2-\frac1p}}\,d\theta.
\end{align}
Here we see that 
\begin{align*}
\frac{d}{d\theta} 
&\left(\frac{-\cos_p^{p-1}{\theta}}{(1-k^p\sin_p^p{\theta})^{1-\frac1p}}\right)\\
&=\frac{(p-1)\sin_p^{p-1}{\theta}(1-k^p\sin_p^p{\theta})
-(p-1)k^p\sin_p^{p-1}{\theta}\cos_p^p{\theta}}
{(1-k^p\sin_p^p{\theta})^{2-\frac1p}}\\
&=\frac{(p-1)(k')^p\sin_p^{p-1}{\theta}}
{(1-k^p\sin_p^p{\theta})^{2-\frac1p}},
\end{align*}
so that we use integration by parts as
\begin{align*}
\frac{dK_p}{dk}
&=\int_0^{\frac{\pi_p}{2}}
\frac{k^{p-1}}{(k')^p}\frac{d}{d\theta} 
\left(\frac{-\cos_p^{p-1}{\theta}}
{(1-k^p\sin_p^p{\theta})^{1-\frac1p}}\right)
\sin_p{\theta}\,d\theta\\
&=\frac{k^{p-1}}{(k')^p} \left[\frac{-\cos_p^{p-1}{\theta}
\sin_p{\theta}}
{(1-k^p\sin_p^p{\theta})^{1-\frac1p}}
\right]_0^{\frac{\pi_p}{2}}
+\frac{k^{p-1}}{(k')^p} \int_0^{\frac{\pi_p}{2}} 
\frac{\cos_p^p{\theta}}{(1-k^p\sin_p^p{\theta})^{1-\frac1p}}
\,d\theta\\
&=\frac{k^{p-1}}{(k')^p} \int_0^{\frac{\pi_p}{2}}
\frac{1}{k^p} \cdot
\frac{1-k^p\sin_p^p{\theta}-(1-k^p)}{(1-k^p\sin_p^p{\theta})^{1-\frac1p}}\,d\theta\\
&=\frac{1}{k(k')^p}(E_p-(k')^pK_p).
\end{align*}
This completes the proof.
\end{proof}


Proposition \ref{prop:p-differential} now yields 
Theorem \ref{thm:p-legendre}.

\begin{proof}[Proof of Theorem \ref{thm:p-legendre}]
We will differentiate the left-hand side of \eqref{eq:p-legendre}
and apply Proposition \ref{prop:p-differential}.
As $dk'/dk=-(k/k')^{p-1}$
we have
\begin{equation}
\frac{dK_p'}{dk}=\frac{k^pK_p'-E_p'}{k(k')^p},
\quad
\frac{dE_p'}{dk}=k^{p-1}\frac{K_p'-E_p'}{(k')^p}.
\label{eq:p-differential'}
\end{equation}
Hence a direct computation shows that
\begin{align*}
\frac{d}{dk}(K_p'E_p & +K_pE_p'-K_pK_p')\\
&=
\frac{k^pK_p'-E_p'}{k(k')^p}\cdot E_p+K_p'\cdot \frac{E_p-K_p}{k}
+\frac{E_p-(k')^pK_p}{k(k')^p}\cdot E_p'\\
& \qquad +K_p\cdot k^{p-1}\frac{K_p'-E_p'}{(k')^p}-\frac{E_p-(k')^pK_p}{k(k')^p}\cdot K_p'-K_p\cdot \frac{k^pK_p'-E_p'}{k(k')^p}\\
&=E_p'E_p
\left(-\frac{1}{k(k')^p}+\frac{1}{k(k')^p}\right)
+(K_p'E_p-K_pE_p')
\left(\frac{k^{p-1}}{(k')^p}+\frac{1}{k}-\frac{1}{k(k')^p}\right)\\
& \qquad +K_p'K_p 
\left(-\frac{1}{k}+\frac{k^{p-1}}{(k')^p}+\frac{1}{k}-\frac{k^{p-1}}{(k')^p}\right)\\
&=0.
\end{align*}
Therefore the left-hand side of \eqref{eq:p-legendre} is a constant $C$. 

We will evaluate $C$ as follows. 
It is easy to see that $\lim_{k \to +0}K_pE_p'=\pi_p/2$. 
Moreover, since
\begin{align*}
|(K_p & -E_p)K_p'| \\
&=\int_0^{\frac{\pi_p}{2}} 
\left(\frac{1}{(1-k^p\sin_p^p{\theta})^{1-\frac1p}}
-(1-k^p\sin_p^p{\theta})^{\frac1p}\right)\,d\theta \cdot
\int_0^{\frac{\pi_p}{2}} 
\frac{d\theta}{(1-(k')^p\sin_p^p{\theta})^{1-\frac1p}}\\
&=\int_0^{\frac{\pi_p}{2}} 
\frac{k^p\sin_p^p{\theta}}{(1-k^p\sin_p^p{\theta})^{1-\frac1p}}
\,d\theta \cdot
\int_0^{\frac{\pi_p}{2}} 
\frac{d\theta}{(\cos_p^p{\theta}+k^p\sin_p^p{\theta})^{1-\frac1p}}\\
& \leq 
kK_p(k)\frac{\pi_p}{2},
\end{align*}
we obtain $\lim_{k \to +0}(K_p-E_p)K_p'=0$. Thus, 
letting $k \to +0$ in the left-hand side of \eqref{eq:p-legendre},
we conclude that $C=\pi_p/2$.
\end{proof}


\begin{prop}
\label{prop:p-hypergeometric}
$K_p(k)$ and $K_p'(k)$ satisfy
$$\frac{d}{dk}\left(k(k')^p \frac{dy}{dk}\right)
=(p-1)k^{p-1}y,$$
that is
$$k(1-k^p)\frac{d^2y}{dk^2}+(1-(p+1)k^p)\frac{dy}{dk}
-(p-1)k^{p-1}y=0.$$
Moreover $E_p(k)$ and $E_p'(k)-K_p'(k)$ satisfy
$$(k')^p\frac{d}{dk}\left(k\frac{dy}{dk}\right)=-k^{p-1}y,$$
that is 
$$k(1-k^p)\frac{d^2y}{dk^2}+(1-k^p)\frac{dy}{dk}
+k^{p-1}y=0.$$
\end{prop}

\begin{proof}
Let us first give proofs for $K_p$ and $E_p$.
Repeated application of Proposition \ref{prop:p-differential}
and $d(k')^p/dk=-pk^{p-1}$ enables us to see that
\begin{align*}
\frac{d}{dk}\left(k(k')^p \frac{dK_p}{dk} \right)
&=\frac{d}{dk}(E_p-(k')^pK_p)\\
&=\frac1k (E_p-K_p)-\left(-pk^{p-1}K_p+\frac1k (E_p-(k')^pK_p)\right)\\
&=(p-1)k^{p-1}K_p
\end{align*}
and
\begin{align*}
\frac{d}{dk}\left(k\frac{dE_p}{dk}\right)
&=\frac{d}{dk}(E_p-K_p)\\
&=\frac1k (E_p-K_p)-\frac{1}{k(k')^p}(E_p-(k')^p K_p)\\
&=-\frac{k^{p-1}}{(k')^p}E_p.
\end{align*}

Similarly, it follows easily from \eqref{eq:p-differential'} that
$K_p'$ satisfies
\begin{align*}
\frac{d}{dk}\left(k(k')^p \frac{dK_p'}{dk}\right)
&=(p-1)k^{p-1}K_p'.
\end{align*}
Set $H_p(k):=E_p(k)-K_p(k)$. 
Using \eqref{eq:p-differential'} repeatedly,
we have $dH_p'/dk=E_p'/k$
and 
$$\frac{d}{dk}\left(k\frac{dH_p'}{dk}\right)
=\frac{dE_p'}{dk}
=-\frac{k^{p-1}}{(k')^p}H_p'.$$
The proof is complete.
\end{proof}


Define $K_p^*(k)$ and $E_p^*(k)$ as conjugates
for $K_p(k)$ and $E_p(k)$ respectively: for $k \in [0,1)$
\begin{align}
K^*_p(k)
&:=\int_0^{\frac{\pi_p}{2}} \frac{d\theta}{(1-k^p\sin_p^p{\theta})^{\frac1p}}
=\int_0^1 \frac{dt}{(1-t^p)^{\frac1p}(1-k^pt^p)^{\frac1p}},\\
E^*_p(k)
&:=\int_0^{\frac{\pi_p}{2}} (1-k^p\sin_p^p{\theta})^{1-\frac1p}\,d\theta
=\int_0^1 \frac{(1-k^pt^p)^{1-\frac1p}}{(1-t^p)^\frac1p}\,dt.
\end{align}
It is clear that $K^*_2(k)=K_2(k)=K(k)$ and $E^*_2(k)=E_2(k)=E(k)$. 
The integral $K_p^*(k)$ appears in the study \cite{T}
for bifurcation problems of $p$-Laplacian.
In \cite{Wa}, $K_p^*(k)$ with $t^p$ replaced by $t^2$
is applied to the planar $p$-elastic problem.

Here are some elementary relations between the integrals
we have introduced.
To state the relations, it is convenient to use the notation 
$i_p:=e^{i\pi/p}$ and for a nonnegative number $\ell$
\begin{align}
K^*_p(i_p\ell)
&:=\int_0^{\frac{\pi_p}{2}} \frac{d\theta}{(1+\ell^p\sin_p^p{\theta})^{\frac1p}}
=\int_0^1 \frac{dt}{(1-t^p)^{\frac1p}(1+\ell^pt^p)^{\frac1p}},\\
E^*_p(i_p\ell)
&:=\int_0^{\frac{\pi_p}{2}} (1+\ell^p\sin_p^p{\theta})^{1-\frac1p}\,d\theta
=\int_0^1 \frac{(1+\ell^pt^p)^{1-\frac1p}}{(1-t^p)^\frac1p}\,dt.
\end{align}

\begin{prop}
\label{prop:Krelation}
Let $0 \leq k<1$ and $k^p+(k')^p=1$. Then
\begin{align}
\label{eq:1}
K_{p^*}(k^{p-1})
&=(p-1)\frac{1}{k'}K^*_p\left(i_p\frac{k}{k'}\right),\\
\label{eq:2}
K_{p^*}(k^{p-1})&=(p-1)K_p(k),\\
\label{eq:3}
E_{p^*}(k^{p-1})
&=(p-1)(k')^{p-1}E^*_p\left(i_p\frac{k}{k'}\right),\\
\label{eq:4}
E_{p^*}(k^{p-1})
&=E_p(k)+(p-2)(k')^pK_p(k),
\end{align}
where $p^*:=p/(p-1)$.
\end{prop}

\begin{proof}
Let us first show \eqref{eq:1} and \eqref{eq:3}. 
Setting $1-t^{p^*}=u^p$ in each integral, we have
\begin{align}
K_{p^*}(k^{p-1})
&=\int_0^1 \frac{dt}{(1-t^{p^*})^{\frac{1}{p^*}}(1-k^pt^{p^*})^{1-\frac{1}{p^*}}}
\notag \\
&=(p-1)\int_0^1 \frac{du}{(1-u^p)^{\frac1p} (1-k^p+k^pu^p)^\frac1p}
\label{eq:kkk} \\
&=(p-1)\frac{1}{k'}K^*_p \left(i_p\frac{k}{k'}\right) \notag
\end{align}
and
\begin{align}
E_{p^*}(k^{p-1})
&=\int_0^1 \left(\frac{1-k^pt^{p^*}}{1-t^{p^*}}\right)^{\frac{1}{p^*}}\,dt
\notag \\
&=(p-1)\int_0^1 \frac{(1-k^p+k^pu^p)^{1-\frac1p}}{(1-u^p)^{\frac1p}}\,du
\label{eq:eee} \\
&=(p-1)(k')^{p-1}E^*_p\left(i_p\frac{k}{k'}\right). \notag
\end{align}

Next we will prove \eqref{eq:2}. Changing the variable as
$$u^p=\frac{(1-k^p)t^p}{1-k^pt^p},$$
in \eqref{eq:kkk} we obtain
\begin{align*}
K_{p^*}(k^{p-1})
&=(p-1)\int_0^1
\frac{\frac{(1-k^p)^\frac1p}{(1-k^pt^p)^{1+\frac1p}}}
{(\frac{1-t^p}{1-k^pt^p})^\frac1p
(\frac{1-k^p}{1-k^pt^p})^\frac1p}\,dt\\
&=(p-1)\int_0^1
\frac{dt}{(1-t^p)^\frac1p (1-k^pt^p)^{1-\frac1p}}\\
&=(p-1)K_p(k).
\end{align*}

To deduce \eqref{eq:4}, we make use of the same 
change of variable above to \eqref{eq:eee}.
\begin{align*}
E_{p^*}(k^{p-1})
&=(p-1)\int_0^1
\frac{(\frac{1-k^p}{1-k^pt^p})^{1-\frac1p}}
{(\frac{1-t^p}{1-k^pt^p})^\frac1p}
\frac{(1-k^p)^\frac1p}{(1-k^pt^p)^{1+\frac1p}}\,dt\\
&=(p-1)(k')^p\int_0^1
\frac{dt}{(1-t^p)^\frac1p (1-k^pt^p)^{2-\frac1p}}.
\end{align*}
In the last integral, after setting $t=\sin_p{\theta}$, 
using \eqref{eq:dK} and Proposition \ref{prop:p-differential}
we have
\begin{align*}
\int_0^1
\frac{dt}{(1-t^p)^\frac1p (1-k^pt^p)^{2-\frac1p}}
&=\int_0^{\frac{\pi_p}{2}} \frac{d\theta}{(1-k^p\sin_p^p{\theta})^{2-\frac1p}}\\
&=\int_0^{\frac{\pi_p}{2}} \frac{d\theta}{(1-k^p\sin_p^p{\theta})^{1-\frac1p}}
+\int_0^{\frac{\pi_p}{2}} \frac{k^p\sin_p^p{\theta}}{(1-k^p\sin_p^p{\theta})^{2-\frac1p}}\,d\theta\\
&=K_p+\frac{k}{p-1}\frac{dK_p}{dk}\\
&=K_p+\frac{1}{(p-1)(k')^p}(E_p-(k')^pK_p).
\end{align*}
Thus
\begin{align*}
E_{p^*}(k^{p-1})
&=(p-1)(k')^p \left(K_p+\frac{1}{(p-1)(k')^p}(E_p-(k')^pK_p)\right)\\
&=E_p+(p-2)(k')^pK_p.
\end{align*}
Therefore we conclude \eqref{eq:4}.
\end{proof}

From Proposition \ref{prop:Krelation} it immediately follows
\begin{cor}
\label{cor:ippanka}
Let $0 \leq k<1$ and $k^p+(k')^p=1$. Then
\begin{align*}
K_p(k)
&=\frac{1}{k'}K_p^*\left(i_p\frac{k}{k'}\right),\\
E_p(k)
&=(k')^{p-1}\left(
(p-1)E_p^*\left(i_p\frac{k}{k'}\right)
-(p-2)K_p^*\left(i_p\frac{k}{k'}\right)
\right).
\end{align*}
\end{cor}

For $p=2$, the identities of Corollary \ref{cor:ippanka} are equivalent to:
$$K(k)=\frac{1}{k'}K\left(i\frac{k}{k'}\right),\quad
E(k)=k'E\left(i\frac{k}{k'}\right),$$
which can be found in \cite[Table 4, p.\,319]{EMOT2}.  

The next corollary means that
$pK_p(\ell^\frac1p)$ and $E_p(\ell^\frac1p)-(1-\ell)K_p(\ell^\frac1p)$ 
have duality properties with 
respect to $p$. 
\begin{cor}
Let $0 \leq \ell <1$. Then 
\begin{align}
\label{eq:2'}
p^*K_{p^*}(\ell^{\frac{1}{p^*}})&=pK_p(\ell^\frac1p),\\
\label{eq:4'}
E_{p^*}(\ell^\frac{1}{p^*})-(1-\ell)K_{p^*}(\ell^\frac{1}{p^*})
&=E_p(\ell^\frac1p)-(1-\ell)K_p(\ell^\frac1p).
\end{align}
\end{cor}

\begin{proof}
Putting $k=\ell^{\frac1p}$ in \eqref{eq:2} and multiplying it by $p^*$,
we obtain \eqref{eq:2'} at once.
To deduce \eqref{eq:4'}, putting $k=\ell^\frac1p$ in \eqref{eq:4} 
and using \eqref{eq:2'} we have 
\begin{align*}
E_{p^*}(\ell^\frac{1}{p^*}) &-(1-\ell)K_{p^*}(\ell^\frac{1}{p^*})\\
&=E_p(\ell^\frac1p)+(p-2)(1-\ell)K_p(\ell^\frac1p)-(1-\ell)(p-1)K_p(\ell^{\frac1p})\\
&=E_p(\ell^\frac1p)-(1-\ell)K_p(\ell^\frac1p).
\end{align*}
Therefore we conclude \eqref{eq:4'}.
\end{proof}

Letting $\ell=0$ in \eqref{eq:2'} we have $p^*\pi_{p^*}=p\pi_p$,
which was indicated in \cite{BE,LE}.

\begin{prop}
\label{prop:2power}
\begin{align*}
K_p\left(\frac{1}{\sqrt[p]{2}}\right)&=\frac{\sqrt[p]{2}\Gamma (\frac{1}{2p})^2}
{4p\Gamma (\frac1p)  \cos{\frac{\pi}{2p}}},\\
E_p\left(\frac{1}{\sqrt[p]{2}}\right)&=\frac{\sqrt[p]{2}}{8p\Gamma(\frac1p)}
\left(\frac{\Gamma(\frac{1}{2p})^2}{\cos{\frac{\pi}{2p}}}
+\frac{2p\Gamma(\frac{1}{2p}+\frac12)^2}{\sin{\frac{\pi}{2p}}}
\right),
\end{align*}
where $\Gamma$ is the gamma function. 
\end{prop}

\begin{proof}
Let $k=1/\sqrt[p]{2}$, then $k'=1/\sqrt[p]{2}$.
By Corollary \ref{cor:ippanka} we have
\begin{align*}
K_p\left(\frac{1}{\sqrt[p]{2}}\right)
=\sqrt[p]{2}K_p^*(i_p)
=\sqrt[p]{2} \int_0^1 \frac{dt}{(1-t^{2p})^\frac1p}.
\end{align*}
Putting $t^{2p}=x$, we obtain
\begin{align*}
\int_0^1 \frac{dt}{(1-t^{2p})^\frac1p}
=\frac{1}{2p} B\left(\frac{1}{2p},1-\frac1p\right)
=\frac{\Gamma(\frac{1}{2p}) \Gamma(1-\frac1p)}
{2p\Gamma (1-\frac{1}{2p})},
\end{align*}
where $B$ is the beta function. 
Since
\begin{align}
\label{eq:Gamma1}
\Gamma \left(1-\frac1p\right)
&=\frac{\pi}{\Gamma(\frac1p)\sin{\frac{\pi}{p}}}
=\frac{\pi}{2\Gamma(\frac1p)\sin{\frac{\pi}{2p}}\cos{\frac{\pi}{2p}}},\\
\label{eq:Gamma2}
\Gamma \left(1-\frac{1}{2p}\right)
&=\frac{\pi}{\Gamma(\frac{1}{2p})\sin{\frac{\pi}{2p}}},
\end{align}
the first formula in the proposition follows.

Similarly, Corollary \ref{cor:ippanka} yields
\begin{align*}
E_{p}\left(\frac{1}{\sqrt[p]{2}}\right)
&=\frac{1}{\sqrt[p^*]{2}}((p-1)E_p^*(i_p)-(p-2)K_p^*(i_p)).
\end{align*}
Since
$$E_p^*(i_p)=\int_0^1\frac{1+t^p}{(1-t^{2p})^{\frac1p}}\,dt
=K_p^*(i_p)+\int_0^1 \frac{t^p}{(1-t^{2p})^{\frac1p}}\,dt,$$
we have
\begin{align*}
E_{p}\left(\frac{1}{\sqrt[p]{2}}\right)
&=\frac{1}{\sqrt[p^*]{2}}
\left(K_p^*(i_p)+(p-1) \int_0^1 \frac{t^p}{(1-t^{2p})^{\frac1p}} \right)\\
&=\frac{1}{2p\sqrt[p^*]{2}}
\left(\frac{\Gamma(\frac{1}{2p})\Gamma(1-\frac1p)}{\Gamma(1-\frac{1}{2p})}
+(p-1)\frac{\Gamma(\frac{1}{2p}+\frac12)\Gamma(1-\frac1p)}
{\Gamma(\frac32-\frac{1}{2p})}\right).
\end{align*}
By \eqref{eq:Gamma1}, \eqref{eq:Gamma2} and
$$\Gamma \left(\frac32-\frac{1}{2p}\right)
=\frac{\pi}{2p^*\Gamma(\frac{1}{2p}+\frac12)\cos{\frac{\pi}{2p}}},$$
the second formula in the proposition follows.
\end{proof}

For $p=2$ Proposition \ref{prop:Krelation} gives that 
$$K\left(\frac{1}{\sqrt{2}}\right)
=\frac{\Gamma(\frac14)^2}{4\sqrt{\pi}},\quad
E\left(\frac{1}{\sqrt{2}}\right)
=\frac{\Gamma(\frac14)^2+4\Gamma(\frac34)^2}{8\sqrt{\pi}},$$
which can be found in \cite[Theorem 1.7]{BB3}. 
See also \cite[Section 13.8]{EMOT2} and \cite[Section 22.8]{WW}.

On account of Proposition \ref{prop:2power}, it is possible to show
$C=\pi_p/2$ in the proof of Theorem \ref{thm:p-legendre} in another way. 
Indeed, letting $k=1/\sqrt[p]{2}$ we have
\begin{align*}
C
&=2K_p\left(\frac{1}{\sqrt[p]{2}}\right)E_p\left(\frac{1}{\sqrt[p]{2}}\right)
-K_p\left(\frac{1}{\sqrt[p]{2}}\right)^2\\
&=\frac{(\sqrt[p]{2})^2\Gamma(\frac{1}{2p})^2}{16p^2\Gamma(\frac1p)^2
\cos{\frac{\pi}{2p}}}
\left(\frac{\Gamma(\frac{1}{2p})^2}{\cos{\frac{\pi}{2p}}}
+\frac{2p\Gamma(\frac{1}{2p}+\frac12)^2}{\sin{\frac{\pi}{2p}}}\right)
-\frac{(\sqrt[p]{2})^2\Gamma(\frac{1}{2p})^4}{16p^2\Gamma(\frac1p)^2
\cos^2{\frac{\pi}{2p}}}\\
&=\frac{(\sqrt[p]{2})^2\Gamma(\frac{1}{2p})^2\Gamma(\frac{1}{2p}+\frac12)^2}
{4p\Gamma(\frac1p)^2\sin{\frac{\pi}{p}}}\\
&=\frac{2^{\frac2p-2}}{p\sin{\frac{\pi}{p}}}
\left(\frac{\sqrt{\pi}}{2^{\frac1p-1}}\right)^2=\frac{\pi_p}{2},
\end{align*}
where we have used the duplication formula 
$$\Gamma(2z)=\frac{2^{2z-1}}{\sqrt{\pi}} \Gamma(z) \Gamma \left(z+\frac12\right)$$
with $z=1/(2p)$.


The remainder of this section will be devoted to the study of relation
between complete $p$-elliptic integrals and hypergeometric series.

For a real number $a$ and a natural number $n$,
we define
$$(a)_n:=\frac{\Gamma(a+n)}{\Gamma(a)}
=(a+n-1)(a+n-2)\cdots (a+1) a.$$
We adopt the convention that $(a)_0:=1$.
For $|x|<1$ the series
$$F(a,b;c;x):=\sum_{n=0}^\infty \frac{(a)_n(b)_n}{(c)_n}\frac{x^n}{n!}$$
is called a \textit{Gaussian hypergeometric series}. 
See \cite{AAR,BB3,EMOT,WW} for more details.

\begin{lem}
\label{lem:integral}
For $n=0,1,2,\ldots$
$$\int_0^{\frac{\pi_p}{2}} \sin_p^{pn}{\theta}\,d\theta
=\frac{\pi_p}{2}\frac{(\frac1p)_n}{n!}.$$
\end{lem}

\begin{proof}
Letting $\sin_p^p{\theta}=t$, we have 
$$
\int_0^{\frac{\pi_p}{2}} \sin_p^{pn}{\theta}\,d\theta
=\frac1p \int_0^1 t^{n+\frac1p-1}(1-t)^{-\frac1p}\,dt
=\frac1p B\left(n+\frac1p,1-\frac1p \right).
$$
Moreover,
\begin{align*}
\frac1p B\left(n+\frac1p,1-\frac1p \right)
&=\frac1p B\left(\frac1p,1-\frac1p \right)
\frac{B\left( n+\frac1p,1-\frac1p \right)}
{B\left(\frac1p,1-\frac1p \right)}\\
&=\frac{\pi_p}{2}\frac{\Gamma(n+\frac1p)}{\Gamma(\frac1p)\Gamma(n+1) }\\
&=\frac{\pi_p}{2}\frac{(\frac1p)_n}{n!},
\end{align*}
and the lemma follows.
\end{proof}

\begin{prop}
\label{prop:hypergeometricexpression}
For $k \in (0,1)$
\begin{align*}
K_p(k)
&=\frac{\pi_p}{2}F\left(\frac1p,1-\frac1p;1;k^p\right),\\
E_p(k)
&=\frac{\pi_p}{2}F\left(\frac1p,-\frac1p;1;k^p\right),\\
K^*_p(k)
&=\frac{\pi_p}{2}F\left(\frac1p,\frac1p;1;k^p\right),\\
E^*_p(k)
&=\frac{\pi_p}{2}F\left(\frac1p,\frac1p-1;1;k^p\right).
\end{align*}
\end{prop}

\begin{proof}
Binomial series expansion gives
\begin{align*}
K_p(k)
=\int_0^{\frac{\pi_p}{2}} 
(1-k^p\sin_p^p{\theta})^{\frac1p-1}\,d\theta
=\sum_{n=0}^\infty (-1)^n \binom{\frac1p-1}{n} k^{pn}
\int_0^{\frac{\pi_p}{2}} \sin_p^{pn}{\theta}\,d\theta.
\end{align*}
Here, using Lemma \ref{lem:integral} and the fact
\begin{align*}
(-1)^n \binom{\frac1p-1}{n}
&=(-1)^n \frac{(\frac1p-1)(\frac1p-2)\cdots(\frac1p-n)}{n!}
=\frac{(1-\frac1p)_n}{(1)_n},
\end{align*}
we see that 
$$K_p(k)=\frac{\pi_p}{2}\sum_{n=0}^\infty
\frac{(\frac1p)_n (1-\frac1p)_n}{(1)_n}\frac{k^{pn}}{n!}
=\frac{\pi_p}{2} F\left(\frac1p,1-\frac1p;1;k^p\right).$$
The other cases are similar and we left to the reader. 
\end{proof}

A hypergeometric series $F(a,b;c;x)$ satisfies 
the hypergeometric differential equation
$$x(1-x)\frac{d^2y}{dx^2}+(c-(a+b+1)x)\frac{dy}{dx}-aby=0.$$
From the fact and Proposition \ref{prop:hypergeometricexpression}
we can also prove Proposition \ref{prop:p-hypergeometric}.

As mentioned in the Introduction, 
the Borweins \cite[Section 5.5]{BB3} define the 
generalized complete elliptic 
integrals of the first kind $\mathrm{K}_s(k)$ and 
of the second kind $\mathrm{E}_s(k)$
by \eqref{eq:Ks} and \eqref{eq:Es} respectively.
They indicate that these functions satisfy
\begin{equation}
\label{eq:legendre-borwein}
\mathrm{K}_s'(k)\mathrm{E}_s(k)+\mathrm{K}_s(k)\mathrm{E}_s'(k)
-\mathrm{K}_s(k)\mathrm{K}_s'(k)=\frac{\pi}{2}\frac{\cos{\pi s}}{1+2s},
\end{equation}
where $\mathrm{K}_s'(k):=\mathrm{K}_s(k'),\
\mathrm{E}_s'(k):=\mathrm{E}_s(k')$ and $k':=\sqrt{1-k^2}$.
Letting $s=1/p-1/2$ in this equality we can also prove Theorem \ref{thm:p-legendre}.

They obtained \eqref{eq:legendre-borwein}, relying on the following identity of
hypergeometric functions with $a=-b=c=s$:
\begin{multline}
\label{eq:elliott}
F\left({{\frac12+a,-\frac12-c}\atop{a+b+1}} ;x\right)
F\left({{\frac12-a,c+\frac12}\atop{b+c+1}} ;1-x\right)\\
+F\left({{a+\frac12,\frac12-c}\atop{a+b+1}};x\right)
F\left({{-(a+\frac12),c+\frac12}\atop{b+c+1}};1-x\right)\\
-F\left({{a+\frac12,\frac12-c}\atop{a+b+1}};x\right)
F\left({{\frac12-a,c+\frac12}\atop{b+c+1}};1-x\right)\\
=\frac{\Gamma(a+b+1)\Gamma(b+c+1)}{\Gamma(a+b+c+\frac32)\Gamma(b+\frac12)},
\end{multline}
which was given by Elliott \cite{El}
(see also \cite{AQVV}, \cite[Theorem 3.2.8]{AAR} and \cite[(13) p.\,85]{EMOT}).
In contrast to this, our approach to Theorem \ref{thm:p-legendre} is more self-contained.

Finally in this section, 
we refer the reader to \cite{AQVV}
for generalized elliptic integrals 
in geometric function theory and the
relationship with hypergeometric functions.
In \cite{AQVV} they also define a generalized Jacobian elliptic function
related to $\mathrm{K}_s(k)$. The author \cite{T2} 
produces generalized Jacobian elliptic functions
with two parameters $p$ and $q$ related to $K_p(k)$, 
but as real functions (cf. \cite{T}).


\section{Application}

In this section, we will apply the complete $p$-elliptic integrals
\eqref{eq:CpEIK} and \eqref{eq:CpEIE}
to compute $\pi_p$, and prove Theorem \ref{thm:main}.
For the very special case of $\pi_p$ with $p=3$,
we are able to obtain a computation formula like \eqref{eq:pi}
for $\pi$ by Salamin \cite{Sa} and Brent \cite{Br}.

Let $a \geq b>0$. Consider the sequences
$\{a_n\}$ and $\{b_n\}$ satisfying $a_0=a, b_0=b$ and 
$$a_{n+1}=\frac{a_n+2b_n}{3}, \quad b_{n+1}=\sqrt[3]{\frac{(a_n^2+a_nb_n+b_n^2)
b_n}{3}},\quad n=0,1,2,\cdots.$$
It is easy to see that $a_n \geq b_n$ for any $n$,
$\{a_n\}$ is decreasing and $\{b_n\}$ is increasing. 
Hence each sequence converges to 
a limit as $n \to \infty$. Moreover, since
\begin{equation}
\label{eq:cubic}
a_{n+1}-b_{n+1} \leq \frac{a_n+2b_n}{3}-b_n = \frac13(a_n-b_n),
\end{equation}
these limits are same. We will denote by $M_3(a,b)$ the common limit
for $a$ and $b$.     

To show Theorem \ref{thm:main}, 
the following identity by Ramanujan for the hypergeometric function
$F(1/3,2/3;1;x)$ is extremely important. 

\begin{lem}[Ramanujan's cubic transformation]
\label{lem:ramanujan}
For $0<k \leq 1$
\begin{equation}
\label{eq:ramanujan}
F\left(\frac13,\frac23;1;1-k^3\right)
=\frac{3}{1+2k}F\left(\frac13,\frac23;1;\left(\frac{1-k}{1+2k}\right)^3\right).
\end{equation}
\end{lem}

\begin{proof}
This identity has been proved by, for instance, 
the Borweins \cite{BB2}, Berndt et al. \cite[Corollary 2.4]{BBG}
or \cite[Corollary 2.4 and (2.25)]{Be}, and Chan \cite{Ch},
though Ramanujan did not leave his proof.
We will present a new proof with elementary calculation.

Since Proposition \ref{prop:hypergeometricexpression} yields
$$K_3(k)=\frac{\pi_3}{2}F\left(\frac13,\frac23;1;k^3\right),$$
so that \eqref{eq:ramanujan} is equivalent to
\begin{equation}
\label{eq:RCT}
K_3(k')=\frac{3}{1+2k}K_3\left(\frac{1-k}{1+2k}\right).
\end{equation}
We have known from Proposition 
\ref{prop:p-hypergeometric} that $K_3(k')$ satisfies
\begin{equation}
\label{eq:hRCT}
\frac{d}{dk}\left(k(k')^3\frac{dy}{dk}\right)=2k^2y.
\end{equation}
To show \eqref{eq:RCT} we will verify that 
the function of right-hand side of \eqref{eq:RCT}
also satisfies \eqref{eq:hRCT}.
Now we let
$$f(k)=\frac{3}{1+2k}K_3\left(\frac{1-k}{1+2k}\right).$$
Applying Proposition \ref{prop:p-differential}
we have
\begin{equation}
\label{eq:dif}
\frac{df(k)}{dk}=\frac{f(k)}{1-k}-\frac{(1+2k)E_3(\ell)}{k(k')^3},
\end{equation}
where $\ell=(1-k)/(1+2k)$.
Thus, differentiating both sides of 
$$k(k')^3\frac{df(k)}{dk}
=(1+k+k^2)kf(k)-(1+2k)E_3(\ell)$$
gives
\begin{multline}
\label{eq:difdif}
\frac{d}{dk}\left(k(k')^3\frac{df(k)}{dk}\right)\\
=(1+2k+3k^2)f(k)+(1+k+k^2)k\frac{df(k)}{dk}
-\left(2E_3(\ell)+(1+2k)\frac{dE_3(\ell)}{dk}\right).
\end{multline}
Here, by Proposition \ref{prop:p-differential}
$$\frac{dE_3(\ell)}{dk}
=\frac{f(k)}{1-k}-\frac{3E_3(\ell)}{(1-k) (1+2k)}.$$
Applying this and \eqref{eq:dif} to \eqref{eq:difdif}, 
we see that the right-hand side of 
\eqref{eq:difdif} is equal to $2k^2f(k)$.  
This shows that $f(k)$ also satisfies \eqref{eq:hRCT} as $K_3(k')$ does.

The equation \eqref{eq:hRCT} has a regular singular point at $k=1$
and the roots of the associated indicial equation are both $0$.
Thus, it follows from the theory of ordinary differential equations 
that the functions $K_3(k')$ and $f(k)$, which agree at $k=1$, must be equal.
This concludes the lemma.
\end{proof}

\begin{prop}
\label{prop:KE}
Let $0 \leq k<1$, then
\begin{enumerate}
\item $K_3(k)=\dfrac{1}{1+2k} K_3\left(\dfrac{\sqrt[3]{9(1+k+k^2)k}}{1+2k}\right)$,
\item $K_3(k)=\dfrac{3}{1+2k'} K_3\left(\dfrac{1-k'}{1+2k'}\right)$,
\item $E_3(k)=\dfrac{1+2k}{3}E_3\left(\dfrac{\sqrt[3]{9(1+k+k^2)k}}{1+2k}\right)
+\dfrac{(1-k)(2+k)}{3}K_3(k)$,
\item $E_3(k)=(1+2k')E_3\left(\dfrac{1-k'}{1+2k'}\right)-k'(1+k')K_3(k)$,
\end{enumerate}
where $k'=\sqrt[3]{1-k^3}$.
\end{prop}

\begin{proof}
It is obvious that (ii) is equivalent to \eqref{eq:RCT},
hence to Lemma \ref{lem:ramanujan}.
Therefore we will show (i), (iv) and (iii) in this order.

(i) Setting in (ii)
$$\frac{1-k'}{1+2k'}=\ell,$$
we get $0 \leq \ell<1$ and 
$$k'=\frac{1-\ell}{1+2\ell},\quad 
k=\frac{\sqrt[3]{9(1+\ell+\ell^2)\ell}}{1+2\ell}.$$
Then (ii) is equivalent to 
$$K_3\left(\frac{\sqrt[3]{9 (1+\ell+\ell^2)\ell}}{1+2\ell}\right)=(1+2\ell)K_3(\ell).$$
Replacing $\ell$ by $k$, we obtain (i).

(iv) Let $\ell$ be the number above, then 
$$\frac{d\ell}{dk}
=\frac{3k^2}{(1+2k')^2(k')^2}.$$
It follows from (ii) that $(1+2k')K_3(k)=3K_3(\ell)$.
Differentiating both sides in $k$, we have
\begin{multline*}
-2\left(\frac{k}{k'}\right)^2K_3(k)+\frac{1+2k'}{k(k')^3}E_3(k)-\frac{1+2k'}{k}K_3(k)\\
=\frac{9k^2}{(1+2k')^2(k')^2} \cdot \frac{1}{\ell (\ell')^3}E_3(\ell)
-\frac{9k^2}{(1+2k')^2(k')^2\ell} K_3(\ell).
\end{multline*}
Applying 
$$\ell=\frac{1-k'}{1+2k'},\quad \ell'=\frac{\sqrt[3]{9(1+k'+(k')^2)k'}}{1+2k'}$$
and (ii), we see that the right-hand side is written as   
$$\frac{(1+2k')^2}{k(k')^3}E_3(\ell)-\frac{3k^2}{(1-k')(k')^2}K_3(k).$$
Thus we have
$$\frac{1+2k'}{k(k')^3}E_3(k)=\frac{(1+2k')^2}{k(k')^3}E_3(\ell)-\frac{(1+2k')(1+k')}{k(k')^2}K_3(k).$$
Multiplying this by $k(k')^3/(1+2k')$, we obtain (iv).

(iii) It is obvious that (iv) can be written in $\ell$, that is,
\begin{multline*}
E_3\left(\frac{\sqrt[3]{9(1+\ell+\ell^2)\ell}}{1+2\ell}\right)\\
=\frac{3}{1+2\ell}E_3(\ell)-\frac{(1-\ell)(2+\ell)}{(1+2\ell)^2}
K_3\left(\frac{\sqrt[3]{9 (1+\ell+\ell^2)\ell}}{1+2\ell}\right).
\end{multline*}
From (i) we have (iii). The proof is complete.
\end{proof}

Let us introduce auxiliary functions 
\begin{align*}
I_p(a,b)
&:=\int_0^{\frac{\pi_p}{2}} \frac{d\theta}{(a^p\cos_p^p{\theta}
+b^p\sin_p^p{\theta})^{1-\frac1p}},\\
J_p(a,b)
&:=\int_0^{\frac{\pi_p}{2}} (a^p\cos_p^p{\theta}
+b^p\sin_p^p{\theta})^{\frac1p}\,d\theta.
\end{align*}
It is easy to check that $K_p(k),\ K_p'(k),\ E_p(k)$ and $E_p'(k)$ are
written as
\begin{align*}
K_p(k)&=I_p(1,k'),\ K_p'(k)=I_p(1,k),\\
E_p(k)&=J_p(1,k'),\ E_p'(k)=J_p(1,k).
\end{align*}

In the remainder of this section \textit{we assume $p=3$}.
We will write $I_3(a,b)$ and $J_3(a,b)$ simplicity
$I(a,b)$ and $J(a,b)$ respectively when no confusion can arise.

The next lemma is crucial to show Theorem \ref{thm:p-agm}.
\begin{lem}
\label{lem:I}
For $a \geq b>0$,
$$aI(a,b)=\frac{a+2b}{3}
I\left(\frac{a+2b}{3},\sqrt[3]{\frac{(a^2+ab+b^2)b}{3}}\right).$$
\end{lem}

\begin{proof}
From Proposition \ref{prop:KE} (ii) (with $k$ replaced by $k'$) we get
\begin{align*}
aI(a,b)
&=
\frac1a K_3'\left(\frac{b}{a}\right)\\
&=\frac{3}{a+2b}K_3\left(\frac{a-b}{a+2b}\right)\\
&=\frac{3}{a+2b}I\left(1,\frac{\sqrt[3]{9(a^2+ab+b^2)b}}{a+2b}\right)\\
&=\frac{a+2b}{3}I\left(\frac{a+2b}{3},\sqrt[3]{\frac{(a^2+ab+b^2)b}{3}}\right).
\end{align*}
This proves the lemma.
\end{proof}

Lemma \ref{lem:I} implies that
$\{a_nI(a_n,b_n)\}$ is a constant sequence:
\begin{equation}
\label{eq:constant}
aI(a,b)=a_1I(a_1,b_1)=a_2I(a_2,b_2)=\cdots=a_nI(a_n,b_n)=\cdots.
\end{equation}
Letting $n \to \infty$ in \eqref{eq:constant} we have

\begin{prop}
\label{prop:KM}
For $a \geq b>0$
$$aI(a,b)=\frac{\pi_3}{2}\frac{1}{M_3(a,b)}.$$
\end{prop}

From above, Theorem \ref{thm:p-agm} immediately follows. 

\begin{proof}[Proof of Theorem \ref{thm:p-agm}]
Put $a=1$ and $b=k'$ in Proposition \ref{prop:KM}.
\end{proof}

Let $I_n:=I(a_n,b_n),\ J_n:=J(a_n,b_n)$, then
\begin{lem}
\label{lem:IJ}
For $a \geq b>0$
$$3J_{n+1}-J_n=a_nb_n(a_n+b_n)I_n, \quad n=0,1,2,\ldots.$$
\end{lem}

\begin{proof}
Set $\kappa_n:=\sqrt[3]{1-(b_n/a_n)^3}$.
We see at once that
\begin{equation}
\label{eq:in}
I_n=\frac{1}{a_n^2}K_3(\kappa_n),\quad
J_n=a_nE_3(\kappa_n),\quad n=0,1,2,\ldots.
\end{equation}

Now, letting $k=\kappa_n$ in Proposition \ref{prop:KE} (iv), we have
$$E_3(\kappa_n)=(1+2\kappa_n')E_3
\left(\frac{1-\kappa_n'}{1+2\kappa_n'}\right)
-\kappa_n'(1+\kappa_n')K_3(\kappa_n).$$
It is easily seen that $\kappa_n'=b_n/a_n$ and $\kappa_{n+1}=(1-\kappa_n')/(1+2\kappa_n')$.
Thus
$$E_3(\kappa_n)=\frac{a_n+2b_n}{a_n}E_3(\kappa_{n+1})
-\frac{b_n}{a_n}\left(1+\frac{b_n}{a_n}\right)K_3(\kappa_n).$$
Multiplying this by $a_n$ and using $a_n+2b_n=3a_{n+1}$ we obtain
$$a_nE_3(\kappa_n)=3a_{n+1}E_3(\kappa_{n+1})-b_n\left(1+\frac{b_n}{a_n}\right)K_3(\kappa_n).$$
From \eqref{eq:in} we accomplished the proof.
\end{proof}

\begin{prop}
\label{prop:EK}
Let $a \geq b>0$, then
$$J(a,b)=\left(a^3-a \sum_{n=1}^\infty 3^n (a_n+c_n)c_n \right) I(a,b),$$
where $c_n:=\sqrt[3]{a_n^3-b_n^3}$.
\end{prop}

\begin{proof}
We denote $I(a,b)$ and $J(a,b)$ briefly by $I$ and $J$ respectively.
Lemma \ref{lem:I} gives
$a_nI_n=aI$ for any $n$. By Lemma \ref{lem:IJ} 
and $c_{n+1}=(a_n-b_n)/3$, we obtain
\begin{align*}
3(J_{n+1}-aa_{n+1}^2I)-(J_n-aa_n^2I)
&=(ab_n(a_n+b_n)-3aa_{n+1}^2+aa_n^2)I\\
&=\frac{a}{3}(2a_n^2-a_nb_n-b_n^2)I\\
&=\frac{a}{3}(2a_n+b_n)(a_n-b_n)I\\
&=3a(a_{n+1}+c_{n+1})c_{n+1}I.
\end{align*}
Multiplying this by $3^n$ and summing both sizes 
from $n=0$ to $n=m-1$, we obtain
\begin{align}
\label{eq:J-aI}
3^{m}(J_{m}-a a_{m}^2I)-(J-a^3I)
&=a\left(\sum_{n=1}^{m} 3^{n} (a_{n}+c_{n})c_{n}\right) I.
\end{align}
On the other hand, since $aI=a_{m}I_{m}$, we have
\begin{align*}
3^{m}(J_{m}-aa_{m}^2I)
&=3^{m}\int_0^{\frac{\pi_3}{2}}
\frac{a_{m}^3\cos_3^3{\theta}+b_{m}^3\sin_3^3{\theta}-a_{m}^3}
{(a_{m}^3\cos_3^3{\theta}+b_{m}^3\sin_3^3{\theta})^{\frac23}}\, d\theta\\
&=3^{m}c_{m}^3\int_0^{\frac{\pi_3}{2}}
\frac{-\sin_3^3{\theta}}
{(a_{m}^3\cos_3^3{\theta}+b_{m}^3\sin_3^3{\theta})^{\frac23}}\, d\theta.
\end{align*}
By \eqref{eq:cubic} we get
$$0 \leq 3^{m}c_{m}^3 \leq \frac{1}{9^{m}}(a-b)^3,$$
which means 
$\lim_{m \to \infty}3^{m}(J_{m}-aa_{m}^2I)=0$. 
Therefore,  as $m \to \infty$ in \eqref{eq:J-aI}
the proposition follows. 
\end{proof}

Now we are in a position to show Theorem \ref{thm:main}.

\begin{proof}[Proof of Theorem \ref{thm:main}]
Let $k=1/\sqrt[3]{2}$ in Theorem \ref{thm:p-legendre},
then 
\begin{equation}
\label{eq:K}
2K_3\left(\frac{1}{\sqrt[3]{2}}\right)E_3\left(\frac{1}{\sqrt[3]{2}}\right)
-K_3\left(\frac{1}{\sqrt[3]{2}}\right)^2=\frac{\pi_3}{2}.
\end{equation}
Letting $a=1$ and $b=1/\sqrt[3]{2}$ in Proposition \ref{prop:EK} we get
$$E_3\left(\frac{1}{\sqrt[3]{2}}\right)
=\left(1-\sum_{n=1}^\infty 3^n (a_n+c_n)c_n \right) 
K_3\left(\frac{1}{\sqrt[3]{2}}\right),$$
where $c_n=\sqrt[3]{a_n^3-b_n^3}$.
Substituting this 
to \eqref{eq:K}, we have
$$\left(2\left(1-\sum_{n=1}^\infty 3^{n}(a_{n}+c_{n})
c_{n} \right)-1\right)
K_3\left(\frac{1}{\sqrt[3]{2}}\right)^2
=\frac{\pi_3}{2}.$$
Finally, applying Theorem \ref{thm:p-agm} with $k=1/\sqrt[3]{2}$
to this, we obtain
$$\left(1-2\sum_{n=1}^\infty 3^{n}(a_{n}+c_{n})c_{n} \right)
\frac{\pi_3^2}{4M_3\left(1,\dfrac{1}{\sqrt[3]{2}}\right)^2}
=\frac{\pi_3}{2}.$$
This leads the result.
\end{proof}

\begin{rem}
In Theorem \ref{thm:p-agm}, we proved the identity
$$K_3(k)=\frac{\pi_3}{2}\frac{1}{M_3(1,k')}.$$
From the fact and Corollary \ref{cor:ippanka} 
we can formally deduce 
$$K_3^*(k)=\frac{\pi_3}{2}\frac{1}{M_3(k',1)},$$
where $M_3(a,b)$ for $0<a \leq b$ is also defined by 
\eqref{eq:sequence} in the same way
as that for $a \geq b>0$. 
This means 
$$F\left(\frac13,\frac13;1;1-k^3\right)=\sqrt[3]{\frac{3}{1+k+k^2}}
F\left(\frac13,\frac13;1;\frac{(1-k)^3}{9(1+k+k^2)}\right),$$
which is reported by \cite{WCJ}.
\end{rem}







\end{document}